\documentclass[runningheads,a4paper]{llncs}

\usepackage{amssymb}
\usepackage{graphicx}
\usepackage{epstopdf}
\usepackage{subfigure}

\usepackage{url}
\urldef{\mailsa}\path|ackerman@sci.haifa.ac.il|
\urldef{\mailsb}\path|keszegh@renyi.hu|
\urldef{\mailsc}\path|vizermate@gmail.com|
\newcommand{\keywords}[1]{\par\addvspace\baselineskip
\noindent\keywordname\enspace\ignorespaces#1}
\newtheorem{observation}[lemma]{Observation}

\begin{document}

\mainmatter  % start of an individual contribution

\title{On the size of planarly connected crossing graphs}


\author{Eyal Ackerman\inst{1} \and Bal\'azs Keszegh\inst{2} \and Mate Vizer\inst{2}}

\institute{Department of Mathematics, Physics, and Computer Science,\\
University of Haifa at Oranim, Tivon 36006, Israel.\\
\mailsa \\
\and Alfr\'ed R\'enyi Institute of Mathematics,
Hungarian Academy of Sciences,\\
H-1053 Budapest, Hungary.\\
\mailsb, \mailsc }


\toctitle{Lecture Notes in Computer Science}
\tocauthor{Authors' Instructions}
\maketitle

\begin{abstract}
We prove that if an $n$-vertex graph $G$ can be drawn in the plane such that
each pair of crossing edges is independent and there is a crossing-free
edge that connects their endpoints, then $G$ has $O(n)$ edges.
Graphs that admit such drawings are related to quasi-planar graphs 
and to maximal $1$-planar and fan-planar graphs.
%A pair of independent and crossing edges in a drawing of a graph is \emph{planarly connected}
%if there is a crossing-free edge that connects endpoints of the crossed edges.
%A graph is a \emph{planarly connected crossing} (PCC) graph, if it admits a drawing
%in which every pair of independent and crossing edges is planarly connected.
%We prove that a PCC graph with $n$ vertices has $O(n)$ edges.
\end{abstract}

\section{Introduction}
\label{sec:Intro}
%%%%%%%%%%%%%%%%%%%%%%%%%%%%%%%%%%%%%%%%%%%%%%%%%%%%%%%%%%%%%%%%%%%%

Throughout this paper we consider graphs with no loops or parallel edges.
A \emph{topological graph} is a graph drawn in the plane with its vertices
as distinct points and its edges as Jordan arcs that connect the corresponding points
and do not contain any other vertex as an interior point.
Every pair of edges in a topological graph has a finite number of intersection points,
each of which is either a vertex that is common to both edges,
or a crossing point at which one edge passes from one side of the other edge to its other side.
% In this work we will be mainly interested in \emph{simple} topological graphs, that is,
% topological graphs in which every pair of edges intersect at most once.
A topological graph is \emph{simple} if every pair of its edges intersect at most once.
A \emph{geometric} graph is a (simple) topological graph in which every edge is a straight-line segment.
If the vertices of a geometric graph are in convex position,
then the graph is a \emph{convex} geometric graph.

Call a pair of independent\footnote{Two edges are \emph{independent} if they do not share a vertex.
Note that in a simple topological graph two crossing edges must be independent.} and crossing edges $e$ and $e'$
in a topological graph $G$ \emph{planarly connected}
if there is a crossing-free edge in $G$ that connects an endpoint of $e$ and an endpoint of $e'$.
A \emph{planarly connected crossing} (PCC for short) topological graph is a topological graph in
which every pair of independent crossing edges is planarly connected.
An abstract graph is a PCC graph if it can be drawn as a topological PCC graph.

Our motivation for studying PCC graphs comes from two examples of topological graphs that satisfy this property:
A graph is \emph{$k$-planar} if it can be drawn as a topological graph in which each edge is crossed at most $k$ times (we call such a topological graph \emph{$k$-plane}).
Suppose that $G$ is an $n$-vertex $1$-planar topological graph with the maximum possible number of edges
(i.e., there is no $n$-vertex $1$-planar graph with more edges than $G$).
Now consider a drawing $D$ of $G$ as a $1$-plane topological graph with the least number of crossings.
Then it is easy to see that $D$ is a simple topological graph.
Moreover, $D$ is a PCC topological graph.
Indeed, if $(u,v)$ and $(w,z)$ are two independent edges that cross at a point $x$ and are not planarly connected,
then we can draw a crossing-free edge $(u,w)$ that consists of the (perturbed) segments $(u,x)$ and $(w,x)$
of $(u,v)$ and $(w,z)$, respectively.
This way we either increase the number of edges in the graph or we are able to replace a crossed
edge with a crossing-free edge and get a $1$-plane drawing of $G$ with less crossings.

Another example for PCC topological graphs are certain drawings of \emph{fan-planar} graphs.
A graph is called \emph{fan-planar} if it can be drawn as a simple topological graph
such that for every edge $e$ all the edges that cross $e$ share a common endpoint on the same side of $e$.
As before, it can be shown (see~\cite[Corollary~1]{KU14}) that such an embedding of a maximum fan-planar graph
with as many crossing-free edges as possible admits a PCC topological graph.

Both $1$-plane topological graphs and fan-planar graphs are sparse, namely,
their maximum number of edges is $4n-8$~\cite{PT97} and $5n-10$~\cite{KU14}, respectively (where $n$ denotes the number of vertices).
Our main result shows that simple PCC topological graphs are always sparse.

\begin{theorem}
\label{thm:main}
Let $G$ be an $n$-vertex topological graph such that
for every two crossing edges $e$ and $e'$ it holds that $e$ and $e'$ are independent 
and there is a crossing-free edge
that connects an endpoint of $e$ and an endpoint of $e'$.
Then $G$ has at most $cn$ edges, where $c$ is an absolute constant.
\end{theorem}

Note that by definition in a simple topological graph every pair of crossing edges must be independent,
therefore, Theorem~\ref{thm:main} holds for PCC simple topological graphs.
We strongly believe that (not necessarily simple) PCC topological graphs also have linearly many edges,
however, our proof currently falls short of showing that.

It follows from Theorem~\ref{thm:main} that $1$-plane and fan-planar graphs have linearly many edges,
however, with a much weaker upper bound than the known ones.
It would be interesting to improve our upper bound and to find the exact maximum size of a PCC (simple) topological graph.
We show that this value is at least $9n-O(1)$ (see Section~\ref{sec:Discussion}),
which implies that not every PCC graph is a (maximum) $1$-plane or fan-planar graph.

PCC graphs are also related to two other classes of topological graphs.
Call a topological graph \emph{$k$-quasi-plane} if it has no $k$ pairwise crossing edges.
According to a well-known and rather old conjecture (see e.g.,~\cite{BMP05,Pa91})
$k$-quasi-plane graphs should have linearly many edges.

\begin{conjecture}\label{conj:k-quasi-plane}
For any integer $k \geq 2$ there is a constant $c_k$ such that
every $n$-vertex $k$-quasi-plane graph has at most $c_kn$ edges.
\end{conjecture}

It is easy to see that if $G$ is a PCC simple topological graph, then $G$ is $9$-quasi-plane:
Suppose for contradiction that $G$ contains a set $E'$ of $9$ pairwise crossing edges and let $V'$ be the set of their endpoints.
Since $G$ is a simple topological graph, no two edges in $E'$ share an endpoint, therefore $|V'|=18$.
Let $G'$ be the subgraph of $G$ that is induced by $V'$ and let $E''$ be the crossing-free edges of $G'$.
Clearly $(V',E'')$ is a plane graph.
Moreover, all the edges in $E'$ must lie in the same face $f$ of this plane graph, since they are pairwise crossing.
It follows that $f$ is incident to every vertex in $V'$ and therefore $(V',E'')$ is an outerplanar graph.
Thus, $|E''| \leq 2\cdot 18-3=33$.
On the other hand, since $G'$ is also PCC and no two edges in $E'$ share an endpoint, it follows that $|E''| \geq {{9}\choose{2}}=36$, a contradiction.

Therefore, Conjecture~\ref{conj:k-quasi-plane}, if true, would immediately imply Theorem~\ref{thm:main} for simple topological graphs.
However, this conjecture was only verified for $k=3$~\cite{AT07,AA*97,PRT06}, for $k=4$~\cite{Ack09},
and (for any $k$) for convex geometric graphs~\cite{CP92}.
For $k \geq 5$ the currently best upper bounds on the size of $n$-vertex $k$-quasi-plane graphs are
$n(\log n)^{O(\log k)}$ by Fox and Pach~\cite{FP12,FP14},
and $O_k(n\log n)$ for simple topological graphs by Suk and Walczak~\cite{SW15}.

\medskip

Another conjecture that implies Theorem~\ref{thm:main} (also for topological graphs that are not necessarily simple)
is related to \emph{grids} in topological graphs.
A \emph{$k$-grid} in a topological graph is a pair of edge subsets $E_1,E_2$ such that $|E_1|=|E_2|=k$,
and every edge in $E_1$ crosses every edge in $E_2$.
Ackerman et al.~\cite{AF*14} proved that every $n$-vertex topological graph that does not contain
a $k$-grid with distinct vertices has at most $O_k(n\log^*n)$ edges and conjectured
that this upper bound can be improved to $O_k(n)$.
It is not hard to show, as before, that a PCC graph does not contain an $8$-grid with distinct vertices.
Therefore, this conjecture, if true, would also imply Theorem~\ref{thm:main}.

\paragraph{Outline.}
We prove Theorem~\ref{thm:main} in the following section.
In Section~\ref{sec:Discussion} we give a lower bound on the maximum size of a PCC simple topological graph,
generalize the notion of planarly connected edges, and conclude with some open problems.

%%%%%%%%%%%%%%%%%%%%%%%%%%%%%%%%%%%%%%%%%%%%%%%%%%%%%%%%%%%%%%%%%%%%
\section{Proof of Theorem~\ref{thm:main}}
\label{sec:proof}
%%%%%%%%%%%%%%%%%%%%%%%%%%%%%%%%%%%%%%%%%%%%%%%%%%%%%%%%%%%%%%%%%%%%

Let $G=(V,E)$ be an $n$-vertex topological graph such that
for every two crossing edges $e$ and $e'$ it holds that $e$ and $e'$ are independent and
there is a crossing-free edge that connects an endpoint of $e$ and an endpoint of $e'$.
Denote by $E' \subseteq E$ the set of crossing-free (planar) edges in $G$,
and by $E''=E \setminus E'$ the set of crossed edges in $G$.
Since $G'=(V,E')$ is a plane graph, we have $|E'| \le 3n$, so it remains to prove that $|E''|=O(n)$.

Let $G'_1=(V_1,E'_1),\ldots,G'_k=(V_k,E'_k)$ be the connected components of the graph $G'$,
and let $E''_{i,j} = \{ (u,v) \in E'' \mid u \in V_i \textrm{ and } v \in V_j \}$.

\begin{lemma}
\label{lem:E''_{i,i}}
$|E''_{i,i}| \leq 96|V_i|$ for $1 \leq i \leq k$.
\end{lemma}

\begin{proof}
Assume without loss of generality that $i=1$ and consider the graph $G'_1$.
Let $f_1,\ldots,f_\ell$ be the faces of the plane graph $G'_1$.
For a face $f_j$, let $V(f_j)$ be the vertices that are incident to $f_j$,
and let $E''(f_j)$ be the edges in $E''_{1,1}$ that lie within $f_j$
(thus, their endpoints are in $V(f_j)$).
Denote by $|f_j|$ the size of $f_j$, that is, the length of the shortest closed walk
that visits every edge on the boundary of $f_j$.
Recall that in the Introduction we argued that a PCC simple topological graph is $9$-quasi-plane.
For the same arguments we have the following observation.

\begin{observation}
\label{obs:9-quasi-planar}
There are no $9$ pairwise crossing edges in $E''(f_j)$.
\end{observation}

% \begin{proof}
% Suppose for contradiction that $(x_1,y_1),\ldots,(x_9,y_9) \in E''(f_j)$ are pairwise independent and crossing edges.
% Thus, for every pair of crossing edges there is a distinct crossing-free edge that connects endpoints of the two crossing edges.
% Therefore, there are ${{9}\choose{2}}=36$ crossing-free edges that connect vertices in $\{x_1,\ldots,x_9,y_1,\ldots,y_9\}$.
% Note that all of these edges lie outside of $f_j$ or on its boundary, and therefore they define an outer-planar graph.
% However, an outer-planar graph with $18$ vertices contains at most $2\cdot 18-3=33$ edges.
% \end{proof}

\begin{proposition}
\label{prop:convex-quasi}
$|E''(f_j)| \leq 16|f_j|$, for $1 \leq j \leq \ell$.
\end{proposition}

\begin{proof}
Define first an auxiliary graph $\hat{G}_j$ as follows.
When traveling along the boundary of $f_j$ in clockwise direction,
we meet every vertex in $V(f_j)$ at least once and possibly several times if
the boundary of $f_j$ is not a simple cycle.
Let $v_1,v_2,\ldots,v_{|f_j|}$ be the list of vertices as they appear along the boundary of $f_j$,
where a new instance of a vertex is introduced whenever a visited vertex is revisited.
The edge set of $\hat{G}_j$ corresponds to $E''(f_j)$, however, we make sure to pick the ``correct'' instance of a vertex in $v_1,v_2,\ldots,v_{|f_j|}$
for a vertex in $V(f_j)$ that was visited more than once when traveling along the boundary of $f_j$
(see Figure~\ref{fig:convex-quasi} for an example).
\begin{figure}
    \centering
    \subfigure[A face $f_j$ of $G'_1$]{\label{fig:f_j}
    {\includegraphics[width=5cm]{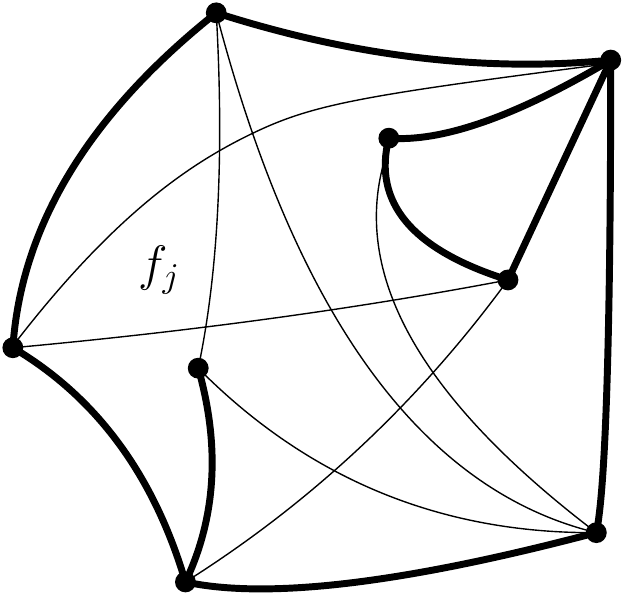}}}
        \hspace{5mm}
    \subfigure[The corresponding graph $\hat{G}_j$.]{\label{fig:G-hat}
    {\includegraphics[width=5cm]{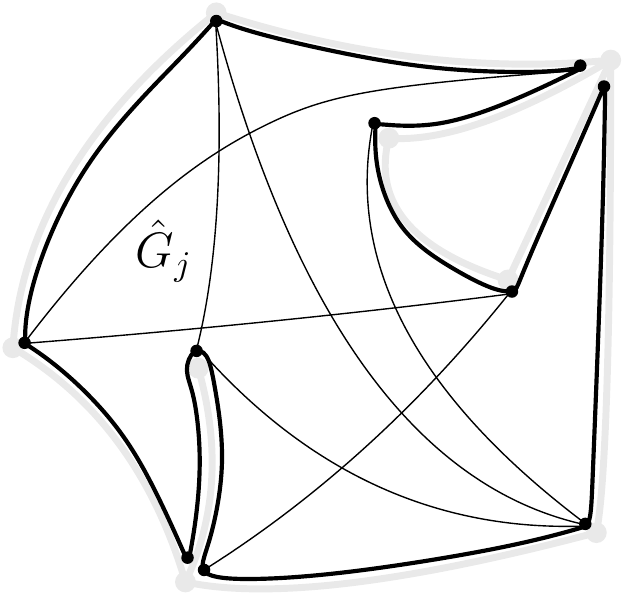}}}
	\caption{Illustrations for the proof of Proposition~\ref{prop:convex-quasi}.}
	\label{fig:convex-quasi}
\end{figure}

Let $\hat{e}_1$ and $\hat{e}_2$ be a pair of crossing edges in $\hat{G}_j$
and let $e_1$ and $e_2$ be their corresponding edges in $G$.
Clearly, $e_1$ and $e_2$ are crossing edges and therefore are independent and planarly connected.
It follows from Observation~\ref{obs:9-quasi-planar} that $\hat{G}_j$ does not contain $9$ pairwise crossing edges.

We now realize the underlying abstract graph of $\hat{G}_j$ as a convex geometric graph:
The vertices $v_1,v_2,\ldots,v_{|f_j|}$ are the vertices of a convex polygon (in that order),
and the edges of $\hat{G}_j$ are realized as straight-line segments.
Suppose that two edges $(v_{i_1},v_{i_2})$ and $(v_{i_3},v_{i_4})$ cross in this realization.
% Then these edges must be independent.
Assume without loss of generality that $i_1 < i_2$, $i_3 < i_4$ and $i_1 < i_3$.
Since these edges are the chords of a convex polygon it must be that $i_1 < i_3 < i_2 < i_4$.
It follows that $(v_{i_1},v_{i_2})$ and $(v_{i_3},v_{i_4})$ also cross in $\hat{G}_j$.
Thus, the realization of $\hat{G}_j$ as a convex geometric graph does not contain $9$ pairwise crossing edges.
% Let $(v_{i_1},v_{i_2})$ and $(v_{i_3},v_{i_4})$ be two independent edges in $\hat{G}_j$.
% Clearly, these edges cross if and only if their corresponding edges cross in $G$.
% Therefore, $\hat{G}_j$ does not contain $9$ pairwise independent and crossing edges by Proposition~\ref{prop:9-quasi-planar}.
% Assume without loss of generality that $i_1 < i_2$, $i_3 < i_4$ and $i_1 < i_3$.
% Then it is not hard to see that $(v_{i_1},v_{i_2})$ and $(v_{i_3},v_{i_4})$ cross an odd number of times if $i_1 < i_3 < i_2 < i_4$
% and an even number of times if $i_1 < i_2 < i_3 < i_4$ or $i_1 < i_3 < i_4 < i_2$.
%
% It is easy to see that two edges in a convex geometric graph cross if and only if their vertices
% alternate in the cyclic order of the vertices along the boundary of the convex hull of the graph.
% It follows that $\hat{G}_j$ can be realized as a convex geometric graph without $9$ pairwise crossing edges.
According to a result of Capoyleas and Pach~\cite{CP92}, an $n$-vertex convex geometric graph
with no $k+1$ pairwise crossing edges has at most ${{n}\choose{2}}$ edges if $n \leq 2k+1$
and at most $2kn - {{2k+1}\choose{2}}$ edges if $n \geq 2k+1$.
Therefore, $|E''(f_j)| \leq 16|f_j|$.
\qed
\end{proof}

We now return to proving that $|E''_{1,1}| = O(|V_1|)$.
Using the fact that $\sum_{j=1}^{\ell} |f_j| = 2|E'_1| \leq 6|V_1|$,
we have $$ |E''_{1,1}| = \sum_{j=1}^{\ell} E''(f_j) \leq \sum_{j=1}^{\ell} 16|f_j| \leq 96|V_1|,$$
which completes the proof of the lemma.
\qed
\end{proof}

It remains to bound the number of edges in $E''$ between different connected components of $G'$.
To this end, we introduce some more notations.
For every $j \neq i$, let $V_{i,j}$ be the vertices of $V_i$ that are connected to some vertex in $V_j$,
i.e., $V_{i,j} = \{ v_i \in V_i \mid (v_i,v_j) \in E'' \textrm{ for some } v_j \in V_j\}$.
Let $H$ be a simple (abstract) graph whose vertex set is $\{u_1,\ldots,u_k\}$
and whose edge set consists of the edges $(u_i,u_j)$ such that $E''_{i,j} \neq \emptyset$.

\begin{lemma}
\label{lem:H-planar}
$H$ is a planar graph.
\end{lemma}

\begin{proof}
For $1 \leq i \leq k$ identify $u_i$ with one of the vertices of $G'_i$ and let $T_i$ be a spanning tree of $G'_i$.
We draw every edge $(u_i,u_j)$ of $H$ as follows:
Pick arbitrarily a pair $v_i \in V_i$ and $v_j \in V_j$ such that $(v_i,v_j) \in E''$.
The edge $(u_i,u_j)$ consists of the unique path in $T_i$ from $u_i$ to $v_i$,
the edge $(v_i,v_j)$ and the unique path in $T_j$ from $v_j$ to $u_j$.
See Figure~\ref{fig:G-and-H} for an example.
\begin{figure}
    \centering
    \subfigure[$G'$ has three connected components.]{\label{fig:G}
    {\includegraphics[width=7cm]{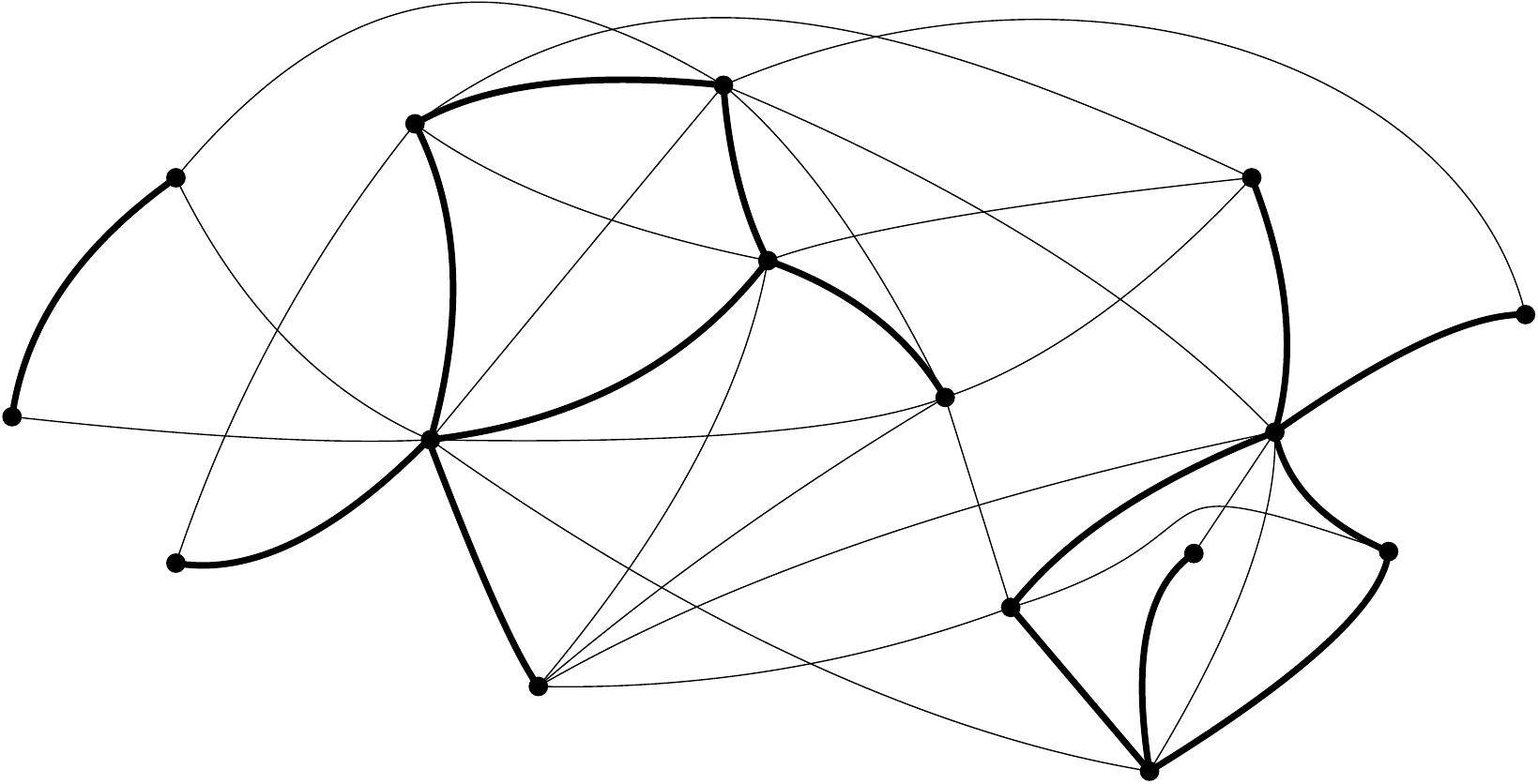}}}
        \hspace{5mm}
    \subfigure[A drawing $H'$ of $H$.]{\label{fig:H}
    {\includegraphics[width=7cm]{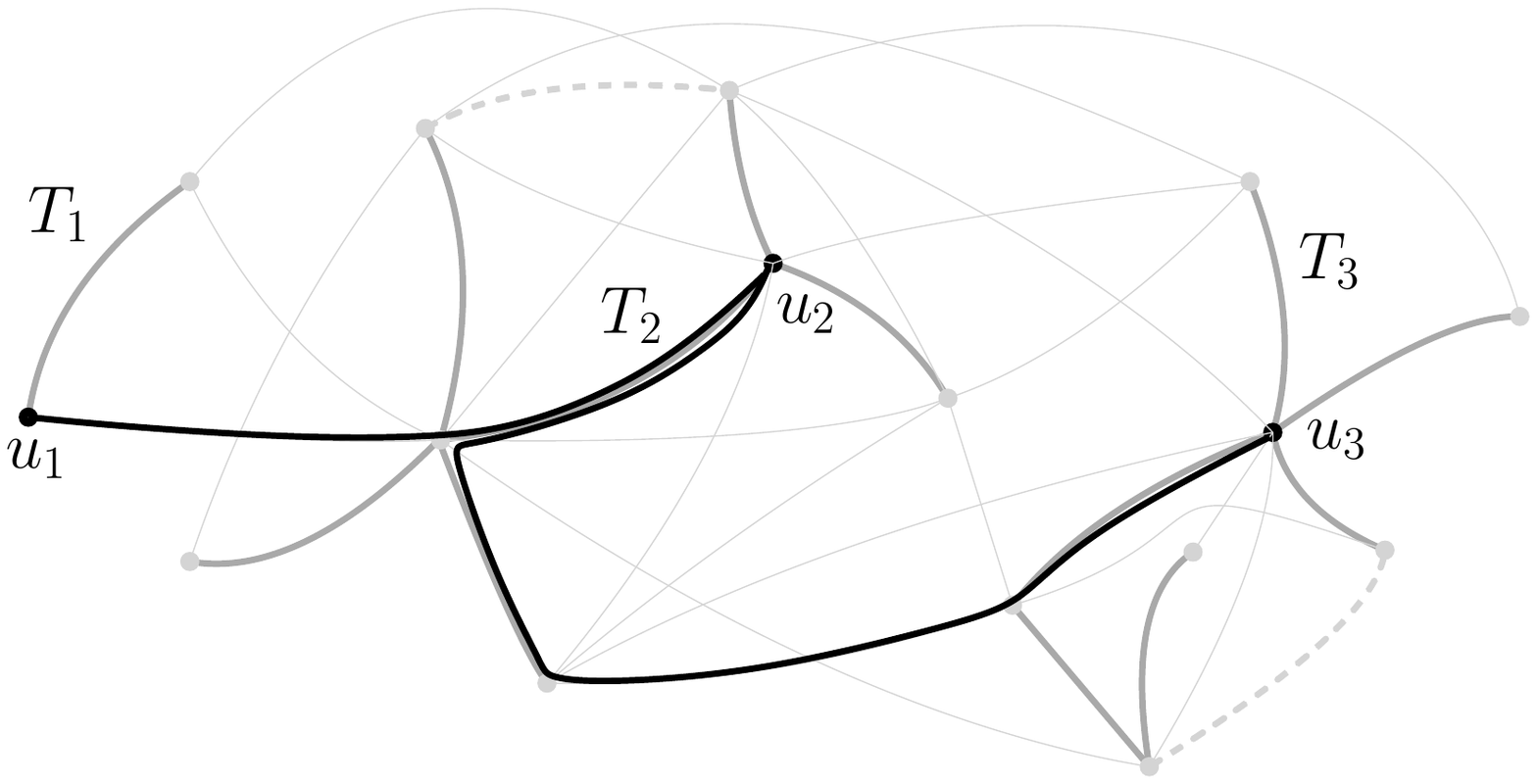}}}
	\caption{Illustrations for the proof of Lemma~\ref{lem:H-planar}.}
	\label{fig:G-and-H}
\end{figure}
Note that in the drawing of $H$ that is obtained this way all the crossing points are inherited from $G$,
however, there are overlaps between edges.
Still, each such (maximal) overlap contains an endpoint of an edge,
and it is not hard to show that the edges in such a drawing
can be slightly perturbed so that all the overlaps are removed and no new crossings are introduced (see~\cite[Lemma~2.4]{AFT12}).
We denote such a drawing of $H$ by $H'$.

The important observation is that if two edges in $H'$ cross, then they must share an endpoint.
Indeed, suppose for contradiction that $(u_a,u_b)$ and $(u_c,u_d)$ are two independent and crossing edges.
Then it follows that $G$ contains two independent and crossing edges $(v_a,v_b)$ and $(v_c,v_d)$, such that $v_a \in V_a$, $v_b \in V_b$, $v_c \in V_c$ and $v_d \in V_d$.
Since these two edges are planarly connected, there should be a crossing-free edge that connects
a vertex in $\{v_a,v_b\}$ with a vertex in $\{v_c,v_d\}$.
However, this is impossible since these four vertices belong to distinct connected components of $G'$.

Finally, a graph that can be drawn so that each crossing is between two edges that share a common vertex is planar:
this follows from the strong Hanani-Tutte Theorem (see, e.g., ~\cite{Ch34,PSS07,T70}).
\qed
\end{proof}

\begin{lemma}
\label{lem:E''_{i,j}}
$|E''_{i,j}| \leq 8(|V_{i,j}|+|V_{j,i}|)$ for every $1 \leq i < j \leq k$.
\end{lemma}

\begin{proof}
Since $G'_i$ and $G'_j$ are planar graphs, we can properly color their vertices with four colors.
Denote the colors by $1,2,3,4$, and let $V_{i,j}^c$ (resp., $V_{j,i}^c$) be the vertices of color $c$ in $V_{i,j}$ (resp., $V_{j,i}$).
We claim that the number of edges in $E''_{i,j}$ that connect a vertex from $V_{i,j}^c$ and a vertex from $V_{j,i}^{c'}$
is at most $2(|V_{i,j}^c|+|V_{j,i}^{c'}|)$ for every $c,c' \in \{1,2,3,4\}$.
Indeed, denote the graph that consists of these edges by $G^*$ and consider its drawing as inherited from $G$.
It is not hard to see that $G^*$ is a planar graph:
Suppose that two edges in $G^*$ cross and denote them by $(u,v)$ and $(x,y)$ such that $u,x \in V_{i,j}^c$ and $v,y \in V_{j,i}^{c'}$.
Since $u$ and $x$ are both of color $c$, there is no crossing-free edge in $G'_i$ that connects them.
Similarly, there is no crossing-free edge in $G'_j$ that connects $v$ and $y$.
Since there are also no crossing-free edges in $E''_{i,j}$, it follows that $(u,v)$ and $(x,y)$ are not independent,
a contradiction.

%Since $G^*$ has no independent crossing edges, it follows that $G^*$ is a planar graph.
Therefore, $G^*$ is a plane graph. Because $G^*$ is also bipartite, its number of edges is at most twice its number of vertices.
Thus, $$|E''_{i,j}| \leq 2\sum_{1 \leq c \leq 4}\sum_{1 \leq c' \leq 4} (|V_{i,j}^c|+|V_{j,i}^{c'}|) = 8(|V_{i,j}|+|V_{j,i}|),$$
and the lemma follows.
\qed
\end{proof}

\begin{lemma}
\label{lem:sum V_{i,j}}
$\sum_{j \neq i} |V_{i,j}| \leq 3(|V_i| + 4\deg_H(u_i))$ for every $1 \leq i \leq k$.
\end{lemma}

\begin{proof}
We use again ideas from the proofs of Lemma~\ref{lem:H-planar} and Lemma~\ref{lem:E''_{i,j}}.
Assume without loss of generality that $i=1$ and consider the graph $G'_1$.
Since $G'_1$ is a planar graph, we can properly color its vertices with four colors.
Denote the colors by $1,2,3,4$, and let $V_1^c$ (resp., $V_{1,j}^c$) be the vertices of color $c$ in $V_1$ (resp., $V_{1,j}$).
Clearly, $\sum_{j=2}^k |V_{1,j}| = \sum_{c=1}^4 \sum_{j=2}^k |V_{1,j}^c|$.
Therefore it is enough to consider $\sum_{j=2}^k |V_{1,j}^c|$ for a fixed color $c$.

Recall that in the proof of Lemma~\ref{lem:H-planar}, for $1 \leq i \leq k$,
we have identified $u_i$ with one of the vertices of $G'_i$ and denoted by $T_i$ a spanning tree of $G'_i$.
We define a graph $H^c$ whose vertex set consists of $V_1^c$ and the vertices $u_j$ that are adjacent to $u_1$ in $H$.
For each such vertex $u_j$ and every vertex $v_1 \in V_{1,j}^c$ % that is connected by at least one edge to a vertex in $V_j$,
pick arbitrarily an edge $(v_1,v_j)$ such that $v_j \in V_j$ (such an edge exists by the definition of $V_{1,j}$), and draw an edge $(v_1,u_j)$ as follows:
$(v_1,u_j)$ consists of the edge $(v_1,v_j)$ in $G$ and the unique path in $T_j$ from $v_j$ to $u_j$.

Observe that $H^c$ is a simple graph (i.e., it has no parallel edges or loops).
Moreover, in the drawing of $H^c$ that is obtained as above, all the crossing points are inherited from $G$,
however, there are overlaps between edges.
Still, each such (maximal) overlap contains an endpoint of an edge,
and thus, as in the proof of Lemma~\ref{lem:H-planar}, the edges of $H^c$
can be slightly perturbed so that all the overlaps are removed and no new crossings are introduced.

Consider such a drawing of $H^c$ and observe that if two edges cross in this drawing, then they must share an endpoint.
Indeed, suppose for contradiction that $(v_1,u_a)$ and $(v'_1,u_b)$ are two independent and crossing edges.
Then $G$ contains two independent and crossing edges $(v_1,v_a)$ and $(v'_1,v_b)$,
such that $v_1,v'_1 \in V_1$, $v_a \in V_a$, and $v_b \in V_b$.
Since these two edges are planarly connected, there should be a crossing-free edge that connects
a vertex in $\{v_1,v_a\}$ with a vertex in $\{v'_1,v_b\}$.
However, this is impossible because there is no crossing-free edge between two vertices from different
connected components of $G'$ and there is also no crossing-free edge $(v_1,v'_1)$ since both $v_1$ and $v'_1$ are of color $c$.

This implies that $H^c$ is a planar graph.
Observe that $\sum_{j=2}^k |V_{1,j}^c|$ is precisely the number of edges in $H^c$.
Thus, $\sum_{j=2}^k |V_{1,j}^c| \leq 3|V(H^c)| = 3(|V_1^c|+\deg_{H}(u_1))$, and it follows that
$\sum_{j=2}^k |V_{1,j}| = \sum_{c=1}^4 \sum_{j=2}^k |V_{1,j}^c| \leq 3|V_1|+12\deg_{H}(u_1)$.
\qed
\end{proof}

Recall that it remains to show that $|E''| = O(n)$:
$$
|E''|= \sum_{1 \leq i \leq k} |E''_{i,i}| + \sum_{1\leq i < j \leq k} |E''_{i,j}|$$
	$$\leq 96n + 8 \sum_{1\leq i < j \leq k} (|V_{i,j}|+|V_{j,i}|)$$
	$$= 96n + 8 \sum_{1\leq i \leq k} \sum_{j \neq i} |V_{i,j}|$$
	$$\leq 96n + 24 \sum_{1\leq i \leq k} (|V_i| + 4\deg_H(u_i))$$
	$$\leq 96n + 24n + 96\cdot 2|E(H)| \leq 120n + 192\cdot 3n = 696n.$$

Note that in the last inequality we used the fact that $H$ is a planar graph.
We conclude that $|E|=|E'|+|E''| \leq 699n$. Theorem~\ref{thm:main} is proved.

%%%%%%%%%%%%%%%%%%%%%%%%%%%%%%%%%%%%%%%%%%%%%%%%%%%%%%%%%%%%%%%%%%%%
\section{Discussion}
\label{sec:Discussion}
%%%%%%%%%%%%%%%%%%%%%%%%%%%%%%%%%%%%%%%%%%%%%%%%%%%%%%%%%%%%%%%%%%%%

Recall that we leave open the question of whether Theorem~\ref{thm:main} holds for PCC topological
graphs in which every pair of crossing edges shares a vertex or is planarly connected.

It would also be interesting to find the maximum size of an $n$-vertex PCC simple topological graph.
The proof of Theorem~\ref{thm:main} shows that this quantity is at most $699n$,
but we believe that a linear bound with a much smaller multiplicative constant holds.
Figure~\ref{fig:geza} describes a construction of an $n$-vertex PCC simple topological graph with $9n-O(1)$ edges.
This construction was given by G\'eza T\'oth~\cite{Geza}, and it improves a construction of ours
with $6.6n-O(1)$ edges that appeared in an earlier version of this paper.
\begin{figure}
    \centering
    \includegraphics[width=6cm]{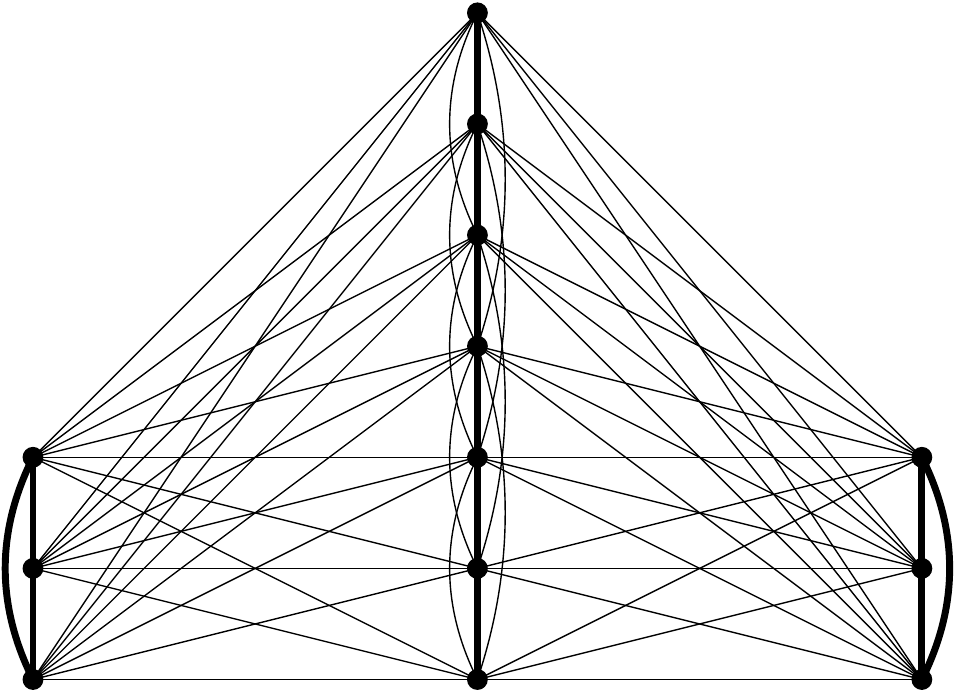}
	\caption{A construction of a topological PCC graph with $9n-O(1)$ edges.}
	\label{fig:geza}
\end{figure}
It goes as follows: place $n-6$ points on the $y$-axis, say at $(0,i)$ for $i=0,1,\ldots,n-7$;
for every $i=0,\ldots,n-8$ add a straight-line edge connecting $(0,i)$ and $(0,i+1)$ (these edges will be crossing-free);
for every $i=0,\ldots,n-9$ add an edge connecting $(0,i)$ and $(0,i+2)$ that goes slightly to the left of the $y$-axis;
for every $i=0,\ldots,n-10$ add an edge connecting $(0,i)$ and $(0,i+3)$ that goes slightly to the right of the $y$-axis;
add three points with the same $x$ coordinate to the left (resp., right) of the $y$-axis and connect each of them by straight-line
edges to each of the points on the $y$-axis;
connect every pair of points to the left (resp., right) of the $y$-axis by a crossing-free edge.
One can easily verify that the resulting graph is indeed a PCC simple topological graph and has $9n-O(1)$ edges.

\medskip

The notion of planarly connected edges can be generalized as follows.
For an integer $k \geq 0$, we say that two crossing edges $e$ and $e'$
in a topological graph $G$ are \emph{$k$-planarly connected}
if there is a path of at most $k$ \emph{crossing-free} edges in $G$
that connects an endpoint of $e$ with an endpoint of $e'$.
Call a graph \emph{$k$-planarly connected crossing} ($k$-PCC for short) graph if it can be drawn
as a topological graph in which every pair of crossing edges is $k$-planarly connected.
Thus, PCC graphs are $1$-PCC graphs.

For $k=0$, graphs that can be drawn as topological graphs in which every
pair of crossing edges share a vertex are actually planar graphs,
as noted in the proof of Lemma~\ref{lem:H-planar}.
For $k \geq 2$ we can no longer claim that a $k$-PCC graph is sparse.
Indeed, it is easy to see that $K_n$ is a $2$-PCC graph:
simply pick a vertex $v$ and draw it with all of its neighbors as a crossing-free star.
Now every remaining edge can be drawn such that we get a simple topological graph
in which for any two crossing edges there is a path (through $v$) of two crossing-free edges that connects
their endpoints.

Note that if $G$ is a $k$-PCC graph and $G'$ is a subgraph of $G$,
then this does not imply that $G'$ is also a $k$-PCC graph.
For example, it is not hard to see that for any $k$ there is a (sparse)
graph that is not $k$-PCC: simply replace every edge of $K_{5}$ (or any non-planar graph)
with a path of length $k+1$.
Call the resulting graph $G'$ and observe that any drawing of it must contain two independent and crossing edges
such that there is no path of length at most $k$ between their endpoints.
On the other hand, if $k \geq 2$ then clearly $G'$ is a subgraph of a $k$-PCC graph ($K_n$).

\medskip

We conclude with a few interesting questions one can ask about the notion of planarly connected crossings:
Is it possible to construct for any $n$ and $k$ a graph with quadratically many edges which is not $k$-PCC?
Can we recognize ($k$-)PCC  graphs efficiently?
Given that a graph is a ($k$-)PCC graph, is it possible to find efficiently such an embedding?

\subsubsection*{Acknowledgments.}
We thank G\'eza T\'oth for his permission to include his construction for a lower bound on the size of a PCC graph in this paper. We also thank an anonymous referee for pointing out an error in an earlier version of this paper.

Most of this work was done during a visit of the first author to the R\'enyi Institute
that was partially supported by the National Research, Development and Innovation Office -- NKFIH under the grant PD 108406 and by the ERC Advanced Research Grant no.\ 267165 (DISCONV). The second author was supported by the National Research, Development and Innovation Office -- NKFIH under the grant PD 108406 and K 116769 and by the J\'anos Bolyai Research Scholarship of the Hungarian Academy of Sciences. The third author was supported by Development and Innovation Office -- NKFIH under the grant SNN 116095.

\bibliography{pcc}

\begin{thebibliography}{10}
\providecommand{\url}[1]{\texttt{#1}}
\providecommand{\urlprefix}{URL }

\bibitem{Ack09}
Ackerman, E.: On the maximum number of edges in topological graphs with no four
  pairwise crossing edges. Discrete {\&} Computational Geometry  41(3),
  365--375 (2009), \url{http://dx.doi.org/10.1007/s00454-009-9143-9}

\bibitem{AF*14}
Ackerman, E., Fox, J., Pach, J., Suk, A.: On grids in topological graphs.
  Comput. Geom.  47(7),  710--723 (2014),
  \url{http://dx.doi.org/10.1016/j.comgeo.2014.02.003}

\bibitem{AFT12}
Ackerman, E., Fulek, R., T{\'{o}}th, C.D.: Graphs that admit polyline drawings
  with few crossing angles. {SIAM} J. Discrete Math.  26(1),  305--320 (2012),
  \url{http://dx.doi.org/10.1137/100819564}

\bibitem{AT07}
Ackerman, E., Tardos, G.: On the maximum number of edges in quasi-planar
  graphs. J. Comb. Theory, Ser. {A}  114(3),  563--571 (2007),
  \url{http://dx.doi.org/10.1016/j.jcta.2006.08.002}

\bibitem{AA*97}
Agarwal, P.K., Aronov, B., Pach, J., Pollack, R., Sharir, M.: Quasi-planar
  graphs have a linear number of edges. Combinatorica  17(1),  1--9 (1997),
  \url{http://dx.doi.org/10.1007/BF01196127}

\bibitem{BMP05}
Brass, P., Moser, W.O.J., Pach, J.: Research Problems in Discrete Geometry.
  Springer (2005)

\bibitem{CP92}
Capoyleas, V., Pach, J.: A {T}ur{\'{a}}n-type theorem on chords of a convex
  polygon. J. Comb. Theory, Ser. {B}  56(1),  9--15 (1992),
  \url{http://dx.doi.org/10.1016/0095-8956(92)90003-G}

\bibitem{Ch34}
Chojnacki, C.: {\"U}ber wesentlich unpl\"attbare {K}urven im dreidimensionalen
  {R}aume. Fundamenta Mathematicae  23(1),  135--142 (1934)

\bibitem{FP12}
Fox, J., Pach, J.: Coloring ${K}_k$-free intersection graphs of geometric
  objects in the plane. Eur. J. Comb.  33(5),  853--866 (2012),
  \url{http://dx.doi.org/10.1016/j.ejc.2011.09.021}

\bibitem{FP14}
Fox, J., Pach, J.: Applications of a new separator theorem for string graphs.
  Combinatorics, Probability {\&} Computing  23(1),  66--74 (2014),
  \url{http://dx.doi.org/10.1017/S0963548313000412}

\bibitem{KU14}
Kaufmann, M., Ueckerdt, T.: The density of fan-planar graphs. CoRR
  abs/1403.6184 (2014), \url{http://arxiv.org/abs/1403.6184}

\bibitem{Pa91}
Pach, J.: Notes on geometric graph theory. In: Goodman, J., Pollack, R.,
  Steiger, W. (eds.) Discrete and Computational Geometry: Papers from DIMACS
  special year, DIMACS series, vol.~6, pp. 273--285. AMS, Providence, RI (1991)

\bibitem{PRT06}
Pach, J., Radoi{\v{c}}i{\'{c}}, R., T{\'o}th, G.: Relaxing planarity for
  topological graphs. In: G{\H{o}}ri, E., Katona, G.O., Lov{\'a}sz, L. (eds.)
  More Graphs, Sets and Numbers, Bolyai Society Mathematical Studies, vol.~15,
  pp. 285--300. Springer, Berlin Heidelberg (2006)

\bibitem{PT97}
Pach, J., T{\'{o}}th, G.: Graphs drawn with few crossings per edge.
  Combinatorica  17(3),  427--439 (1997),
  \url{http://dx.doi.org/10.1007/BF01215922}

\bibitem{PSS07}
Pelsmajer, M.J., Schaefer, M., {\v{S}}tefankovi{\v{c}}, D.: Removing even
  crossings. J. Comb. Theory, Ser. {B}  97(4),  489--500 (2007),
  \url{http://dx.doi.org/10.1016/j.jctb.2006.08.001}

\bibitem{SW15}
Suk, A., Walczak, B.: New bounds on the maximum number of edges in
  k-quasi-planar graphs. Comput. Geom.  50,  24--33 (2015),
  \url{http://dx.doi.org/10.1016/j.comgeo.2015.06.001}

\bibitem{Geza}
T{\'o}th, G.: private communication (2015)

\bibitem{T70}
Tutte, W.: Toward a theory of crossing numbers. Journal of Combinatorial Theory
   8(1),  45--53 (1970),
  \url{http://www.sciencedirect.com/science/article/pii/S0021980070800072}

\end{thebibliography}
\bibliographystyle{splncs03}

\end{document}